\documentclass[11pt,letterpaper]{article}
\usepackage[letterpaper,margin=1in]{geometry}
\usepackage[american]{babel}

\usepackage{algorithm}
\usepackage[noend]{algpseudocode}

\usepackage{dsfont}
\usepackage{amsthm}
\usepackage{amsmath}
\usepackage{amssymb} 
\usepackage{xspace}
\usepackage[usenames,dvipsnames]{xcolor}
\usepackage{enumerate}

\usepackage[textsize=tiny]{todonotes}

\usepackage{thmtools, thm-restate}
\declaretheorem[numberwithin=section]{theorem}

\declaretheorem[sibling=theorem]{lemma}

\declaretheorem[sibling=theorem]{remark}

\usepackage{multirow}

\usepackage[latin1]{inputenc}

\newcommand{\N}{\mathbb{N}}

\renewcommand{\epsilon}{\ensuremath{\varepsilon}}

\begin{document}

\title{Colorability saturation games}
\author{Ralph Keusch \vspace{0.2cm}\\ \small Institute of Theoretical Computer Science\\
\small ETH Zurich, 8092 Zurich, Switzerland \vspace{0.2cm}\\ \small \texttt{rkeusch@inf.ethz.ch}}
\maketitle

\begin{abstract}
We consider the following two-player game: Maxi and Mini start with the empty graph on $n$ vertices and take turns, always adding one additional edge to the graph such that the chromatic number of the current graph is at most $k$, where $k \in \N$ is a given parameter. The game is over when the graph is saturated and no further edge can be inserted. Maxi wants to maximize the length of the game whereas Mini wants to minimize it. The score $s(n,\chi_{>k})$ denotes the number of edges in the final graph, given that both players followed an optimal strategy.

This colorability game belongs to the family of \emph{saturation games} that are known for providing beautiful and challenging problems, despite being defined via simple combinatorial rules. The analysis of colorability saturation games has been initiated recently by Hefetz, Krivelevich, Naor, and Stojakovi{\'c} \cite{hefetz2016saturation}. In this paper, we improve their results by providing almost matching lower and upper bounds on the score of the game for arbitrary choices of $k$ and $n>k$. In addition, we study the specific game with $k=4$ in more details and prove that its score is $n^2/3+O(n)$. 
\end{abstract}

\section{Introduction}\label{sec:saturationinto}

One of the most classic problems in extremal graph theory is to determine how many edges a graph on $n$ vertices can have without fulfilling a given monotone property $\mathcal{P}$. In this context, we say that a graph $G$ is \emph{saturated} with respect to $\mathcal{P}$ if $G$ does not satisfy $\mathcal{P}$, but adding any additional edge $e \in \tbinom{[n]}{2} \setminus E$ to the graph results in $G \cup \{e\}$ satisfying $\mathcal{P}$. The Tur\'{a}n number $ex(n,\mathcal{P})$ is then the \emph{maximal} number of edges that a $\mathcal{P}$-saturated graph on $n$ vertices can have. On the other hand, the saturation number $sat(n,\mathcal{P})$ denotes the \emph{minimal} number of edges that a graph $G$ on $n$ vertices can have while being saturated w.r.t.\ $\mathcal{P}$. For a general survey on saturation numbers see \cite{faudree2011survey}.

Saturation games are a class of combinatorial games that are closely related to saturated graphs. For a given monotone graph property $\mathcal{P}$, the saturation game is played as follows: two players Maxi and Mini start with the empty graph on $n$ vertices. They take turns, always extending the current graph $G$ with some additional edge $e$ such that $G \cup \{e\}$ does not satisfy $\mathcal{P}$. At some point, every free edge is forbidden, i.e., the obtained graph $G_{end}$ is saturated w.r.t.\ $\mathcal{P}$, and the game stops. Mini aims to minimize the number of edges in $G_{end}$ (that is, Mini wants that the game is over as soon as possible), while Maxi's goal is to maximize the number of edges in $G_{end}$. The score of the game, denoted by $s(n,\mathcal{P})$, is the total number of edges in $G_{end}$ when both players apply optimal strategies. When analyzing saturation games, we aim at finding significant lower and upper bounds on the score or ideally determining the score exactly. 

In general, the score can depend on the identity of the first player. However, in this paper we don't specify who starts as all statements hold for both cases. Clearly, for every monotone property $\mathcal{P}$ we have
\[sat(n,\mathcal{P}) \le s(n,\mathcal{P}) \le ex(n,\mathcal{P}),\]
which connects saturation games to the well-studied saturation and Tur\'{a}n numbers of graphs. We see that if the saturation and the Tur\'{a}n numbers of the studied property are the same (for example if $\mathcal{P}=$``being non-planar'' or $\mathcal{P}=$``having independence number at most $k$''), the score $s(n,\mathcal{P})$ is directly determined by the two numbers.

In the last quarter-century, it turned out that analyzing saturation games is both interesting and challenging. The two players not only want to follow their own strategy and play against the adversary \emph{at the same game}, but moreover the two players also have \emph{opposing} goals, thus making the game both intriguing and intricate. Note that the goal of a player is to create a certain graph structure which ensures a short (resp.\ long) game. The more extreme this structure is, the easier the opponent can play against it. But the weaker the structure is, the less our player gains. Often, it is much easier to destroy the opponent's structure than to create the own, desired structure. Hence a good strategy should (a) be resistant against attacks, and (b) make sure that $G_{end}$ will be sufficiently sparse (resp.\ dense). Therefore, finding optimal or almost optimal strategies for saturation games can be surprisingly hard and often requires tedious case distinctions. This is why the asymptotic value of the score is only known for a few particular games.

Let us briefly summarize the most important examples and existing results. Let $\mathcal{C}_k$ be the property of being $k$-connected and spanning, and let $\mathcal{PM}$ be the property of possessing a perfect matching. Carraher, Kinnersley, Reiniger, and West proved $s(n,\mathcal{C}_1)=\tbinom{n-2}{2}+1$ for the connectivity game \cite{carraher2017game} with $n \ge 6$. Hefetz et al.\ \cite{hefetz2016saturation} generalized this result and asserted the bound $s(n,\mathcal{C}_k) \ge \tbinom{n}{2}-5kn^{3/2}$. In the same paper, they proved $s(n,\mathcal{PM})\ge \tbinom{n-4}{2}$ and further results on matching games. Additional saturation games have been studied by Lee and Riet \cite{lee2015fsaturation}, as well as variants on directed graphs \cite{lee2014new}. 

The most famous example of saturation games is the triangle-free game. Here, the considered monotone property is $\mathcal{K}_3$, i.e., containing a triangle as a subgraph. It is well-known that $ex(n,\mathcal{K}_3) = \lfloor n^2/4 \rfloor$ and $sat(n,\mathcal{K}_3)=n-1$ (see \cite{erdos1964problem}, e.g.). In \cite{fueredi1991hajnal} and \cite{seress1992hajnal}, F\"{u}redi, Reimer, and Seress proved a lower bound of $(\frac12 +o(1))n \log_2 n$ on the score of this game, and cite Erd\H{o}s who has given an upper bound of $n^2/5$ in personal communication. However, the proof of this upper bound is lost and could not be retrieved until today. The currently best-known upper bound is $\frac{26}{121}n^2+o(n^2)$ by Bir\'{o}, Horn and Wildstrom \cite{biro2016upper}, a small improvement compared to the trivial upper bound. Closing the large gap between $\Omega(n \log n)$ and $O(n^2)$ is still a challenging open problem, and the current understanding of this game is rather poor.

We now turn to the topic of this paper and focus on the property
\[\chi_{>k}=\text{``having chromatic number at least }k+1\text{''}.\] 
In other words, when playing edges, Mini and Maxi are forced to keep the current graph $k$-colorable. Note that every graph that is saturated w.r.t.\ $\chi_{>k}$ is a complete $k$-partite graph. Hence, the game is about deciding the number of edges of this final, $k$-partite graph $G_{end}$. Clearly, the total number of edges in a complete $k$-partite graph $G$ with partition sizes $n_1, \ldots, n_k$ is $\binom{n}{2}-\sum_{i=1}^k \binom{n_i}{2}$. Then the well-known Tur\'{a}n number 
\[ex(n,\chi_{>k})=(1-1/k+o(1))\binom{n}{2}\] 
and the saturation number
\[sat(n,\chi_{>k}) = (k-1)(n-1)-\binom{k-1}{2}\]
provide us first bounds on the score of the colorability saturation game.

Let us start by describing the case $k=2$ where Maxi and Mini are forced to keep the graph bipartite. If $n$ is even, it is not difficult to observe that Maxi can play such that after each of her moves, every component of the current bipartite graph is balanced, except the isolated vertices, and $G_{end}$ will be perfectly balanced. For the general case, this argument implies
\[s(n,\chi_{>2})=ex(n,\chi_{>2})=\Big\lfloor \frac{n}{4}\Big\rfloor.\]
A formal proof is provided in \cite{carraher2017game}. We see that Mini has no power in this particular saturation game. However, things get more interesting and involved as soon as $k>2$. Hefetz, Krivelevich, Naor, and Stojakovi{\'c} \cite{hefetz2016saturation} proved \[s(n,\chi_{>3}) \le \frac{21}{64}n^2+O(n),\] revealing that Mini now has some influence on the game process. Furthermore, in the same paper they introduced a randomized strategy for Maxi that leads to a general lower bound on the score of colorability saturation games.

\begin{theorem}[Theorem~1.5 in \cite{hefetz2016saturation}] \label{thm:saturationhefetz} There exists a constant $C>0$ such that for every $k \in \N$ and every $n \in \N$ that is sufficiently large compared to $k$ it holds
\[s(n,\chi_{>k})\ge\binom{n}{2}\Big(1-\frac{C \log k}{k}\Big).\]
\end{theorem}

Unfortunately, the proof of this result requires a relatively large constant $C$, making the lower bound trivial for small choices of $k$ (e.g., $k \le 10^4$). 

As main contribution of this paper, we provide almost matching lower and upper bounds on the score $s(n,\chi_{>k})$ which also demonstrate how $s(n,\chi_{>k})$ asymptotically depends on the parameter $k$. In contrast to Theorem~\ref{thm:saturationhefetz}, our results are also non-trivial for small choices of $k$ and therefore enhance the intuitive understanding of the game process. The first result is a general lower bound on the score and improves Theorem~\ref{thm:saturationhefetz}.

\begin{theorem}\label{thm:saturationlower}
Let $k \ge 3$ and $n>k$. Then 
\[s(n,\chi_{>k}) \ge \binom{n}{2}\left(1-\frac{1}{\lceil k/2 \rceil}\right) \ge \binom{n}{2}\left(1-\frac{2}{k}\right).\]
\end{theorem}

Note that for $k=k(n)$ and $n \rightarrow \infty$, Theorem~\ref{thm:saturationlower} and the Tur\'{a}n number $ex(n,\chi_{>k})$ together imply
\[s(n,\chi_{>k}) = \binom{n}{2}-\Theta\Big(\frac{n^2}{k}\Big).\]

Next, we provide a general upper bound which proves that if Mini follows an optimal strategy, the number of missing edges at the end of the game is by a constant factor larger than in a balanced complete $k$-partite graph.

\begin{theorem}\label{thm:saturationupper}
Let $k\ge 4$ and $n>k$. Then 
\[s(n,\chi_{>k}) \le \binom{n}{2} \Big(1-\frac{1}{k-\lfloor (k-1)/3 \rfloor}\Big)+n \le \binom{n}{2} \Big(1-\frac{3}{2k+3}\Big)+n.\]
In particular, if $k$ is fixed and $n\rightarrow \infty$, then 
\[s(n,\chi_{>k}) \le \binom{n}{2}\Big(1-\frac{3}{2k+3}+o(1)\Big).\]
\end{theorem}

The provided lower and upper bounds on $s(n,\chi_{>k})$ are matching up to a small constant factor in the term that counts the missing edges of the final graph. It remains an interesting problem to determine the correct constant. 

In addition, we investigate the specific game with parameter $k=4$ where it turns out that the upper bound given by Theorem~\ref{thm:saturationupper} is tight.

\begin{theorem}\label{thm:4color}
Let $n\ge 5$. Then $s(n,\chi_{>4})=n^2/3+O(n)$.
\end{theorem}

We prove Theorem~\ref{thm:saturationlower} and Theorem~\ref{thm:saturationupper} by using carefully chosen potential functions that are closely related to the density of induced subgraphs. We then define the strategies in terms of these potentials in a general, abstract way such that we need to deal only with a reasonable number of case distinctions. This is a novel approach for the analysis of saturation games.

We start by introducing notations and describing general aspects of our proofs strategies on a high level in Section~\ref{sec:saturationpreliminiaries}. Then in Section~\ref{sec:saturationlower} we provide a general strategy for Maxi and use it to prove Theorem~\ref{thm:saturationlower}. In Section~\ref{sec:saturationupper} we investigate the game from Mini's perspective and show Theorem~\ref{thm:saturationupper}. Afterwards we study the special case $k=4$ in Section~\ref{sec:saturation4color} and prove Theorem~\ref{thm:4color} by using a more specific strategy for Maxi. Finally the last section contains some concluding remarks and open problems.

\section{Preliminaries}\label{sec:saturationpreliminiaries}

Let $k$ and $n$ be two integers such that $n>k$. We study the saturation game on a set $V$ of $n$ vertices w.r.t.\ the monotone property $\chi_{>k}$. The game is considered as a process evolving in time where $G(t)=(V,E_t)$ denotes the graph at the moment where Mini and Maxi have played $t$ edges in total. Note that we start with the empty graph $G(0)$. The game stops at time $t_{end}$ and thus ends with a graph $G_{end}=G(t_{end})$ which is a complete $k$-partite graph. It turns out that in all our proofs, the game is partitioned into two phases. When analyzing the game from the perspective of one specific player, we have a first phase in which our player wants to create a certain graph structure that is suitable for her goal of forcing $G_{end}$ to be either sparse or dense. Once the desired graph structure is present, we enter the second phase where we allow our player to play arbitrarily until the graph is saturated.

We will define the strategies via potential functions. Informally speaking, we measure the progress of ``our'' player by a function $f:V \times \N_0 \rightarrow \N_0$. Then the progress of the player in a set $A \subseteq V$ at time $t$ is given by 
\[f(A,t) := \sum_{v \in A} f(v,t).\] 
We do not yet specify $f$ as the choice of the concrete function depends on the identity of the player and on further notations. In order to quantify the progress of the opponent, we introduce the following notation. Let $A$ and $B$ be two disjoint subsets of $V$ and let $t$ be an integer. We define
\begin{equation}\label{eq:phidef}
\phi(A,B,t)=|E_t[A]|+|E_t[A,B]|.
\end{equation} 
That is, we count the number of edges in the graph $G(t)$ that are either contained in the subgraph induced by $A$ or in the cut between $A$ and $B$. Suppose there exists a vertex set $B$ on which our player has already created her desired structure. Then it turns out that for a set $A \subseteq V \setminus B$, the function $\phi(A,B,t)$ is suitable for measuring the progress of the opponent on set $A$. The goal of our player is now the following: for every subset $A \subseteq V \setminus B$, her own pace should be at least as fast as the pace of her opponent. Hence, she aims to play such that at time $t$, $f(A,t) \ge \phi(A,B,t)$ holds for all $A \subseteq V \setminus B$.

After introducing the most important notations, we provide a criterion for a graph being $k$-colorable. Recall that the $k$-core of a graph $G$ denotes the largest induced subgraph where every vertex has degree at least $k$. Clearly, if the $k$-core is $k$-colorable, then $G$ itself is $k$-colorable because we can take a proper vertex coloring of the $k$-core and extend it vertex by vertex to the whole graph. The following easy lemma is a small modification of this fact. We use it in our proofs whenever we want to verify that an edge proposed by a strategy can indeed be inserted to the graph without violating the colorability constraint.

\begin{lemma} \label{lem:coloring}
Let $k \in \N$, let $G=(V,E)$ be a graph, and let $B \subseteq V$. Suppose that the induced subgraph $G[B]$ is $k$-colorable and that for every non-empty subset $A \subseteq V \setminus B$ it holds $2|E[A]|+|E[A,B]| < k \cdot |A|$. Then $G$ is $k$-colorable.
\end{lemma}

\begin{proof}Let $A' := V \setminus B$. By assumption it holds 
\[\sum_{v \in A'} \deg(v) = 2|E[A']|+|E[A',B]|<k \cdot |A'|.\] 
Hence there exists a vertex $v_1 \in A'$ with degree at most $k-1$. Next, we apply the same argument for the set $A'' := V \setminus (B \cup \{v_1\})$, and afterwards we iterate the argument for all remaining vertices to find an ordering of $V$ where all vertices of $A'$ have back-degree at most $k-1$. Hence, if $G[B]$ is $k$-colorable, we can take an arbitrary vertex coloring of $G[B]$, use the ordering of the vertices, and extend the coloring vertex by vertex to the whole graph $G$ since there is always at least one available color.
\end{proof}

\section{Lower Bound}\label{sec:saturationlower}

In this section we prove Theorem~\ref{thm:saturationlower}. We thus provide a strategy for Maxi that ensures that the game process lasts sufficiently long and $G_{end}$ becomes sufficiently dense. The main idea is the following: Maxi aims to create a collection of vertex-disjoint cliques that cover the entire vertex set $V$. Then every independent set of $G_{end}$ can contain at most one vertex per clique, yielding an upper bound on the independence number of $G_{end}$ and thus a lower bound on the number of edges of the complete $k$-partite graph $G_{end}$. Theorem~\ref{thm:saturationlower} follows directly from the following lemma.

\begin{lemma} \label{lem:cliques}
Let $n \in \N$, $m \ge 2$, and $k \ge 2m-1$. Then in the colorability saturation game with $k$ colors and $n$ vertices, Maxi has a strategy such that in $G_{end}$, the vertex set is covered by $\lceil \frac{n}{m} \rceil$ vertex-disjoint cliques, all having size at most $m$.
\end{lemma}
\begin{proof}[Proof of Theorem~\ref{thm:saturationlower}.]
Let $k\ge 3$, $n>k$, and put $m:=\lceil\frac{k}{2} \rceil$. $G_{end}$ is a complete $k$-partite graph and contains all $\binom{n}{2}$ possible edges except those where both incident vertices are contained in the same partition. By Lemma~\ref{lem:cliques}, Maxi has a strategy such that there are $\lceil \frac{n}{m} \rceil$ vertex-disjoint cliques as induced subgraphs, covering $V$ completely, where every clique contains at most $m$ vertices. Clearly, all vertices of such a clique belong to different color classes of the saturated graph $G_{end}$, therefore every independent set of $G_{end}$ has size at most $\lceil \frac{n}{m} \rceil$.

We claim that the number of missing edges is maximal if there are $\frac{n-k}{\lceil \frac{n}{m}\rceil-1}$ color classes of size $\lceil \frac{n}{m} \rceil$ and $k-\frac{n-k}{\lceil \frac{n}{m}\rceil-1}$ classes of size $1$. Indeed, for every other configuration we could move one vertex from a smaller color class to a larger color class and thereby increase the number of forbidden edges. Hence, the total number of missing edges in $G_{end}$ is at most
\[\frac{n-k}{\big\lceil \frac{n}{m}\big\rceil-1} \cdot \binom{\big\lceil\frac{n}{m}\big\rceil}{2}=\frac{n-k}{2}\Big\lceil\frac{n}{m}\Big\rceil \le \frac{(n-k)(n+m)}{2m} \le \binom{n}{2}\frac{1}{m},\]
proving the statement.
\end{proof}

Before the game starts, Maxi partitions the vertex set $V$ into \emph{disjoint} sets $V_1, \ldots, V_{\lceil\frac{n}{m}\rceil}$, where $|V_i|=m$ holds for all $i \le \frac{n}{m}$. Her goal is to play such that in $G_{end}$, each group $V_i$ induces a clique of size $|V_i|$. For every $1 \le i \le \lceil\frac{n}{m}\rceil$, every $v \in V_i$, and every $t \in \N_0$ we define 
\[\alpha(v,t) := |V_i \cap \Gamma_t(v)|,\] 
where $\Gamma_t(v)$ denotes the neighborhood of $v$ in $G(t)$. Furthermore, for every $A \subseteq V$ and every $t \in \N_0$ we put $\alpha(A,t) := \sum_{v \in A} \alpha(v,t)$. Clearly, for all $v \in V$ we start with $\alpha(v,0)=0$, before the $\alpha$-values start to increase during the game process. At time $t$, we call a vertex $v \in V_i$ \emph{full} if $\alpha(v,t)=|V_i|-1$, i.e., if $v$ is connected to all other vertices of its set $V_i$. 

We want to use the function $\phi(A,B,t)$ as defined in \eqref{eq:phidef} to measure the progress of Mini in subsets $A \subseteq V \setminus B$. The plan is to compare the $\alpha$-values with the $\phi$-values in order to determine where Maxi should insert her next edge. For all points in time $t \in \N_0$ and for all $B \subseteq V$ we put
\[\mathcal{D}(B,t) := \big\{A \subseteq V \setminus B \mid \alpha(A,t) < \phi(A,B,t)\big\}.\]
Informally speaking, the set $\mathcal{D}(B,t)$ contains all \emph{dangerous} subsets $A$ where Mini made more progress than Maxi until time $t$. Clearly, for all choices of $B$ we start with $\mathcal{D}(B,0)=\emptyset$. Finally, we put
\[A_0(B,t) := \bigcap_{A \in \mathcal{D}(B,t)} A.\]

So far, we didn't specify how we pick the set $B$. The concrete choice of $B$ depends on the game process and is quite subtle. We start with $B(0)=0$. Afterwards, for all points in time $t$ where Mini is playing we let $B(t)=B(t-1)$. Whenever Maxi is about to play at time $t$, before her turn we first define the set $B(t)$ according to the following rule.

\begin{enumerate}
\item[(U1)] If all vertices of $V \setminus B(t-1)$ are full in $G(t-1)$, we put $B(t) := V$. If $V \setminus B(t-1)$ contains non-full vertices, $\mathcal{D}(B(t-1),t-1)$ is non-empty, and all vertices of $A_0(B(t-1),t-1)$ are full in $G(t-1)$, we put $B(t) := B(t-1) \cup A_0(B(t-1),t-1)$.  In all other cases, we let $B(t) := B(t-1)$.
\end{enumerate}

Clearly, for all points in time $t$ the set $B(t)$ only contains vertices that are full in $G(t-1)$. (But not necessarily all of them!) As we will see later, it turns out that the rule (U1) guarantees that either the set $\mathcal{D}(B(t),t-1)$ is empty or $ A_0(B(t),t-1)$ contains vertices that are not yet full. We now continue by providing Maxi's strategy for playing her edge at time $t$. 

\begin{enumerate}
\item[(S1)] If $B(t)=V$ but $G(t-1)$ is not yet saturated, insert an arbitrary edge such that $G(t)$ is $k$-colorable.
\item[(S2)] If $\mathcal{D}(B(t),t-1)$ is empty and $B(t) \neq V$, let $v \in V \setminus B(t)$ be a vertex which is not yet full in $G(t-1)$. Insert a new edge $\{u,v\}$, where we require that $u$ is contained in the same group $V_i$ as $v$.
\item[(S3)] If $\mathcal{D}(B(t),t-1)$ is non-empty, let $v \in A_0(B(t),t-1)$ be a vertex which is not yet full in $G(t-1)$. Insert a new edge $\{u,v\}$, where we require that $u$ is contained in the same group $V_i$ as $v$.
\end{enumerate}

A priori, it is not evident that this strategy is well-defined. Amongst others, we have to verify that the rules (S1)-(S3) cover all cases and that the set $A_0(B(t),t-1)$ considered in (S3) is non-empty. When proving Lemma~\ref{lem:cliques} below we show at the same time that the proposed strategy is indeed well-defined. Before starting with the proof, we first state one additional technical lemma whose proof is deferred to the end of this section.

\begin{lemma} \label{lem:dangeroussets}
Let $f:V\rightarrow \N_0$ be a function, let $B \subset V$, and for all $A \subseteq V \setminus B$ and $t \in \N_0$ let $\phi(A,B,t)$ be defined as in \eqref{eq:phidef}. Let $t_0 \in \N_0$ be a point in time of the game process such that $G(t_0)$ is not saturated and such that for all $A \subseteq V \setminus B$ it holds $f(A) := \sum_{v \in A} f(v) \ge \phi(A,B,t_0)$. Then either the set 
\begin{equation}\label{eq:Doflemma}
\mathcal{C} := \big\{A \subseteq V \setminus B \mid f(A) < \phi(A,B,t_0+1) \big\}
\end{equation}
is empty, or the following statements are true.
\begin{enumerate}[(i)]
\item $A' := \cap_{A \in \mathcal{D}} A$ is non-empty and itself contained in $\mathcal{C}$.
\item For all $A \in \mathcal{C}$ it holds $\phi(A,B,t_0+1)=f(A)+1.$
\item For all $A \subseteq V \setminus (B \cup A')$ we have
\[f(A) \ge \phi(A \cup A',B,t_0+1)-\phi(A',B,t_0+1).\]
\end{enumerate}
\end{lemma}

\begin{proof}[Proof of Lemma~\ref{lem:cliques}.]
Let $n \in \N$, let $m \ge 2$, and let $k \ge 2m-1$. Suppose Maxi applies the proposed strategy. We prove by induction that as long as $G(t-1)$ is not saturated, for Maxi's move at time $t$ the following invariants hold.
\begin{enumerate}
\item[(I1)] Exactly one rule of (S1)-(S3) can be applied.
\item[(I2)] Let Maxi insert the desired edge, regardless whether $G(t)$ is $k$-colorable or not. Then the set $\mathcal{D}(B(t),t)$ is empty.
\item[(I3)] Maxi can play her edge without violating the colorability constraint.
\item[(I4)] If there is a non-full vertex $v \in V \setminus B(t)$ in the graph $G(t)$, then $G(t)$ is not yet saturated.
\end{enumerate}

Recall that we start with $B(0)=\emptyset$ and $\mathcal{D}(B(0),0)=\emptyset$.  Let $t$ be a point in time such that $G(t-1)$ is not saturated and assume that either it is Maxi's first move (providing the base case) or that by induction (I1)-(I4) were true for all previous moves of Maxi. We first check property (I1). If every vertex is full in $G(t-1)$, then by (U1) we have $B(t)=V$ and Maxi applies (S1) for the remainder of the game. So let us assume that not every vertex is full in $G(t-1)$. If the set $\mathcal{D}(B(t-1),t-1)$ is empty, we have $B(t)=B(t-1)$ and (S2) is matching. So we can assume that $\mathcal{D}(B(t-1),t-1)$ is non-empty. Then $t>1$, and either by induction we have $\mathcal{D}(B(t-2),t-2)=\emptyset$, or $t=2$ and $\mathcal{D}(B(0),0)=\emptyset$. Hence we can apply Lemma~\ref{lem:dangeroussets} with $f(v)=\alpha(v,t-2)$, $B=B(t-2)$, and $t_0=t-2$. For these choices, the set $\mathcal{C}$ as defined in \eqref{eq:Doflemma} contains all subsets $A \subseteq V \setminus B(t-2)$ such that $\alpha(A,t-2) < \phi(A,B(t-2),t-1)$. Since $\alpha(v,t-1) \ge \alpha(v,t-2)$ holds for every vertex $v \in V$, we have 
\[ \emptyset \neq \mathcal{D}(B(t-1),t-1) \subseteq \mathcal{C}.\]
So $\mathcal{C}$ is non-empty too. By statement (i) of Lemma~\ref{lem:dangeroussets}, $A' = \cap_{A \in \mathcal{C}}$ is non-empty and itself a member of $\mathcal{C}$, i.e., $\alpha(A',t-2) < \phi(A',B(t-2),t-1)$. Combining this fact with the assumption $\mathcal{D}(B(t-1),t-1)\neq \emptyset$ and with statement (ii) of Lemma~\ref{lem:dangeroussets}, we see that for all $v \in A'$ it holds $\alpha(v,t-1)=\alpha(v,t-2)$. Hence $A' \in \mathcal{D}(B(t-1),t-1)$ and, moreover, $A'=A_0(B(t-1),t-1)$. Then the update rule (U1) is well-defined.

When applying (U1), we define the set $B(t)$. In case $B(t)=B(t-1)$, we have $\mathcal{D}(B(t),t-1)=\mathcal{D}(B(t-1),t-1)$ and $A_0(B(t),t-1)=A_0(B(t-1),t-1)$, so (S3) can be applied. It remains the case where $\mathcal{D}(B(t-1),t-1)$ is non-empty and all vertices $v \in A_0(B(t-1),t-1)=A'$ are full in $G(t-1)$ but there exist still non-full vertices. By Lemma~\ref{lem:dangeroussets}~(iii), for all $A \subseteq V \setminus (B(t-1) \cup A')$ we deduce
\[\alpha(A,t-1) \ge \alpha(A,t-2) \ge \phi(A \cup A',B(t-1),t-1)-\phi(A',B(t-1),t-1).\]
Hence, after setting $B(t) = B(t-1) \cup A_0(B(t-1),t-1)$, for every set $A \subseteq V \setminus B(t)$ it holds $\alpha(A,t-1) \ge \phi(A,B(t),t-1)$. So $\mathcal{D}(B(t),t-1)$ becomes empty and (S2) can be applied. This proves invariant (I1).

For invariant (I2), we only have to consider (S2) and (S3), because (S1) is applied when it already holds $B(t)=V$. Maxi now inserts the edge $e=\{u,v\}$. In case she applies (S2), we already know that $\mathcal{D}(B(t),t-1)$ is empty. Clearly we have $u,v \in V \setminus B(t)$. Observe that for a set $A \subseteq V \setminus B(t)$, $\phi(A,B(t),t)>\phi(A,B(t),t-1)$  is only possible when $u,v \in A$. But then, $\phi(A,B(t),t)=\phi(A,B(t),t-1)+1$, $\alpha(u,t)=\alpha(u,t-1)+1$ and $\alpha(v,t)=\alpha(v,t-1)+1$, and we deduce $A \notin \mathcal{D}(B(t),t)$.

Now suppose Maxi applies rule (S3). We have seen before that whenever $B(t) \neq B(t-1)$ it holds $\mathcal{D}(B(t),t-1)=\emptyset$, hence Maxi only uses (S3) in situations where $B(t)=B(t-1)=B(t-2)$. For all sets $A \notin \mathcal{D}(B(t),t-1)$, by the same arguments as for rule (S2) it follows $A \notin \mathcal{D}(B(t),t)$. So we only have to check the sets $A \in \mathcal{D}(B(t),t-1)$. Whenever Maxi uses rule (S3), we have $t >1$, thus by induction $\mathcal{D}(B(t-2),t-2)=\emptyset$. Let us apply again Lemma~\ref{lem:dangeroussets} with the same parameters as before. By statement (ii) of the lemma, for all $A \notin \mathcal{D}(B(t),t-1)$ we have
\begin{equation}\label{eq:howcritical}
\alpha(A,t-1) \ge \alpha(A,t-2) \ge \phi(A,B(t-2),t-1)-1=\phi(A,B(t-1),t-1)-1.
\end{equation}
Furthermore, recall that the set $A'$ considered in Lemma~\ref{lem:dangeroussets} is the same set as $A_0(B(t),t-1)$, and $\mathcal{C}$ is a superset of $\mathcal{D}(B(t-1),t-1)$. Let $A \in \mathcal{D}(B(t),t-1)$. By definition of (S3), Maxi plays such that at least one vertex of the edge $e$ is contained in $A_0(B(t),t-1) \subseteq A$.  We now distinguish two cases. If both $u,v \in A$, then 
\[\alpha(u,t)+\alpha(v,t)=\alpha(u,t-1)+\alpha(v,t-1)+2.\]
Since $\phi(A,B(t),t)\le \phi(A,B(t-1),t-1)+1$, together with \eqref{eq:howcritical} we deduce that $A \notin \mathcal{D}(B(t),t)$. On the other hand, if $u \in A$ but $v \notin A$, the $\phi$-value of the set $A$ does not increase with the new edge $e$, because $v \notin B(t)$ as every vertex of $B(t)$ is full in $G(t-1)$. At the same time, the $\alpha$-value of $u$ increases by one. Again it follows $A \notin \mathcal{D}(B(t),t)$. All together we see that indeed, $\mathcal{D}(B(t),t)$ is empty.

We proceed with invariant (I3). For the rule (S1) it is obvious that $G(t)$ is $k$-colorable. Otherwise, we observe that since Maxi plays her edge $e$ inside a group $V_i$, we have $e \in E[V \setminus B(t)]$ because every vertex of $B(t)$ was already full in $G(t-1)$. Thus in $G(t)$, at least the subgraph induced by $B(t)$ is $k$-colorable. Now, since $\mathcal{D}(B(t),t)$ is empty by (I2), for every non-empty set $A \subseteq V \setminus B(t)$ we have
\begin{equation}\label{eq:fromDtocoloring}
2|E_t[A]|+|E_t[A,B(t)]| \le 2\phi(A,B(t),t) \le 2\alpha(A,t)  \le 2(m-1) |A|<(2m-1)|A|.
\end{equation}
By Lemma~\ref{lem:coloring} and by $k \ge 2m-1$, we see that $G(t)$ is $k$-colorable and indeed, Maxi is allowed to play the desired edge.

It remains the last invariant (I4). Suppose there exists a vertex $v \in V \setminus B(t)$ which is not full in $G(t)$. Then $v \in V_i$ for some index $i$, and there exists another vertex $u \in V_i$ such that $\{u,v\} \notin E_t$. We now fictitiously assume that Mini plays the edge $f=\{u,v\}$ in her move at time $t+1$ and then verify that $G(t+1)$ would be $k$-colorable, implying in turn that $G(t)$ could not be saturated. 

By (I2) it holds $\mathcal{D}(B(t),t)=\emptyset$. We see that if Mini inserts edge $f$, she in fact applies herself rule (S2)! Using similar arguments as above when analyzing rule (S2), we see that for all set $A \subseteq V \setminus B(t)$ where the $\phi$-value increases by one due to the edge $f$, the $\alpha$-value increases too. Therefore, $\mathcal{D}(B(t+1),t+1)$ is empty, given that Mini plays $f$. Then \eqref{eq:fromDtocoloring} for $t+1$ instead of $t$ and Lemma~\ref{lem:coloring} together imply that $G(t+1)$ would be $k$-colorable. Hence (I4) is also true.

We are now able to finish the proof as follows. Suppose Maxi applies the proposed strategy and consider her turn at some point in time $t$ where she answers to Mini's previous turn at time $t-1$. In order to create the desired collection of cliques, it is sufficient to play such that every vertex becomes full. So assume that after potentially applying (U1), the set $V \setminus B(t)$ contains non-full vertices in $G(t-1)$. Since the strategy is well-defined, Maxi answers by applying either rule (S2) or rule (S3). In both cases, by (I3) she can do so such that $G(t)$ is $k$-colorable. Hence, $G(t-1)$ was \emph{not} saturated and we see that if the game stops directly after a move of Mini, then every vertex is full in $G(t-1)$ which is fine. On the other hand, by invariant (I4) Maxi can make $G(t)$ only saturated if every vertex of $V \setminus B(t)$ is full in $G(t)$. Since all vertices of $B(t)$ are already full, we see that if the game stops after a move of Maxi, again \emph{all} vertices are full in $G(t)$. We see that indeed, in the graph $G_{end}$ every subset $V_i$ induced a clique of size $|V_i|$.
\end{proof}

\begin{proof}[Proof of Lemma~\ref{lem:dangeroussets}.]
Let $t_0$ be a point in time of the game process that satisfies the two preconditions of the statement. Suppose that the set $\mathcal{C}$ is non-empty. We first observe that the empty set is not an element of $\mathcal{C}$ as $\phi(\emptyset,B,t_0+1)=0$. Next, assume that $\mathcal{C}$ is non-empty and let $A_1,A_2 \in \mathcal{C}$. We claim that $A_1 \cap A_2 \in \mathcal{C}$. By assumption, we have $\sum_{v \in A_1}f(v) \ge \phi(A_1,B,t_0)$. $G(t_0)$ is not saturated, so there is a player who inserts a new edge $e=\{x,y\}$ at time $t_0+1$. Since $A_1 \in \mathcal{C}$, we have
\[f(A_1) \le \phi(A_1,B,t_0+1)-1 \le \phi(A_1,B,t_0)\le f(A_1),\]
implying 
\begin{equation}\label{eq:A1becamebad}
f(A_1) = \phi(A_1,B,t_0).
\end{equation}
Obviously, $A_2$ achieves the same property. It follows
\begin{align*}
\phi(A_1,B,t_0) + \phi(A_2,B,t_0)  & =f(A_1)+f(A_2)=f(A_1 \cup A_2)+f(A_1 \cap A_2)\\&\ge \phi(A_1 \cup A_2,B,t_0)+\phi(A_1 \cap A_2,B,t_0)\\&\ge \phi(A_1,B,t_0) + \phi(A_2,B,t_0),
\end{align*}
where the first inequality follows by the first assumption on $G(t_0)$ and second inequality follows from the fact that every edge of the graph $G(t_0)$ is counted in  $\phi(A_1 \cup A_2,B,t_0)+\phi(A_1 \cap A_2,B,t_0)$ at least as often as in $\phi(A_1,B,t_0) + \phi(A_2,B,t_0)$, which can be observed by a simple case analysis. However, the above inequality chain implies that we have equality everywhere, and in particular
\begin{equation}\label{eq:cutA1A2}
f(A_1 \cap A_2)=\phi(A_1 \cap A_2,B,t_0).
\end{equation}

From \eqref{eq:A1becamebad} we know that $\phi(A_1,B,t_0+1)=\phi(A_1,B,t_0)+1$, and the same is true for $A_2$. Then either both vertices $x,y$ of the new edge are contained in $A_1 \cap A_2$, or one endpoint is in $A_1 \cap A_2$ and the other in $B$. We see that in both cases, the new edge $e$ contributes to $\phi(A_1 \cap A_2, B,t_0+1)$. Together with \eqref{eq:cutA1A2} we deduce $A_1 \cap A_2 \in \mathcal{C}$. Now that whole argument can be repeated for any two sets $A_1,A_2 \in \mathcal{C}$, and we conclude that $A' = \cap_{A \in \mathcal{C}} A$ is itself in the family $\mathcal{C}$. This proves (i).

Next we observe that statement (ii) follows directly from the property $\phi(A,B,t_0+1)\le \phi(A,B,t_0)+1$ and from \eqref{eq:A1becamebad}. Regarding (iii), this observation implies $\phi(A',B,t_0)=f(A')$. Let $A \subseteq (B \cup \setminus A')$ and recall that  by assumption we have $f(A \cup A') \ge \phi(A \cup A',B,t_0)$. Putting things together, we then arrive at 
\begin{equation}\label{eq:ineqforii}
\phi(A \cup A',B,t_0)-\phi(A',B,t_0) \le f(A \cup A') - f(A') = f(A).
\end{equation}
However, we observe that
\[\phi(A \cup A',B,t_0)-\phi(A',B,t_0)=|E_{t_0}[A]|+|E_{t_0}[A,C \cup A']|.\]
Since the edge $e$ uses at least one vertex of $A'$, it is neither contained in $E_{t_0+1}[A]$ nor in $E_{t_0+1}[A,C \cup A']$. This implies
\[\phi(A \cup A',B,t_0)-\phi(A',B,t_0)=\phi(A \cup A',B,t_0+1)-\phi(A',B,t_0+1).\]
Combining this equality with inequality~\eqref{eq:ineqforii} then proves (iii).
\end{proof}

\section{Upper Bound}\label{sec:saturationupper}

We prove Theorem~\ref{thm:saturationupper} by describing and analyzing a strategy for Mini which shortens the game such that $G_{end}$ becomes sufficiently sparse. Here, the main idea is to play such that in $G_{end}$ there are many vertices of degree $n-1$ (``star vertices''). Then, the color classes of the complete $k$-partite graph $G_{end}$ are rather unbalanced, making $G_{end}$ sparser. Our general upper bound on the score of colorability saturation games follows from the following lemma.

\begin{lemma}\label{lem:stars}
Let $n,\ell \in \N$ and let $k \ge 3\ell+1$. Then in the colorability saturation game with $k$ colors and $n$ vertices, Mini has a strategy such that there are at least $\ell$ vertices of degree $n-1$ in $G_{end}$.
\end{lemma}

\begin{proof}[Proof of Theorem~\ref{thm:saturationupper}.]
Let $k \ge 4$, $n>k$, and put $\ell := \lfloor \frac{k-1}{3}\rfloor$. No matter how Mini and Maxi play, the graph $G_{end}$ is a complete $k$-partite graph with partition sizes $n_1, \ldots, n_k$, containing $\binom{n}{2}-\sum_{i=1}^k \binom{n_i}{2}$ edges. W.l.o.g.\ assume $n_1 \ge \ldots \ge n_k$. By Lemma~\ref{lem:cliques}, Mini has a strategy such that there are $\ell$ vertices of degree $n-1$ in $G_{end}$, i.e., $n_{k-\ell+1}=\ldots=n_k=1$.

We want to lower-bound the number of missing edges in $G_{end}$. Aside from the $\ell$ color classes of size $1$, there are $k-\ell$ color classes left over among which we have to distribute the $n-\ell$ remaining vertices. Then the number of missing edges becomes minimal if for all $1 \le i \le k-\ell$ it holds $\lfloor \frac{n-\ell}{k-\ell}\rfloor \le n_i \le \lceil \frac{n-\ell}{k-\ell}\rceil$. Indeed, for any other distribution we would have $n_1 > \lceil \frac{n-\ell}{k-\ell}\rceil$ and $n_{k-\ell} < \lfloor \frac{n-\ell}{k-\ell}\rfloor$, and transferring one vertex from the first class to class $k-\ell$ would decrease the total number of missing edges. Hence, it follows that the number of edges that are missing in $G_{end}$ is at least
\[(k-\ell) \binom{\frac{n-\ell}{k-\ell}}{2} = \frac{(n-\ell)(n-k)}{2(k-\ell)} \ge \frac{1}{k-\ell}\binom{n}{2} - \frac{n(k+\ell)}{2(k-\ell)}.\]
Furthermore, $\frac{k}{3} \ge \ell \ge \frac{k-3}{3}$ by our choice of $\ell$, therefore $\frac{1}{k-\ell}  \ge \frac{3}{2k+3}$ and $\frac{k+\ell}{k-\ell} \le 2$. We see that the number of missing edges is at least 
\[\frac{1}{k-\ell}\binom{n}{2} - n \ge \frac{3}{2k+3}\binom{n}{2} - n,\]
which proves the theorem.
\end{proof}

In Section~\ref{sec:saturationlower}, we analyzed a strategy where Maxi creates a collection of disjoint cliques. Even though Mini's strategy is different, we use similar proof techniques to verify Lemma~\ref{lem:stars}. There will be one major difference: while Maxi defined \emph{before} the start of the game which vertex sets should become cliques, the strategy that we propose for Mini is more adaptive in the sense that she does not announce at the start which nodes should be come stars, but she chooses these distinguished vertices carefully at specific moments \emph{during} the game process. In the analysis, we therefore use a set $S(t)$, containing all vertices that have been designated until time $t$ to become star vertices. The set $S(t)$ will be increasing in $t$.

For every vertex $v \in V$, every set $S \subseteq V$, and all $t \in \N_{0}$ we define
\[\alpha(v,S,t) := |S \cap \Gamma_t(v)|\]
for measuring the progress of Mini at vertex $v$. Moreover, for all $A \subseteq V$, $S \subseteq V$, and $t \in \N_0$ let $\alpha(A,S,t) := \sum_{v \in A}\alpha(v,S,t)$. For all choices of $v$ and $S$, we start with $\alpha(v,S,0)=0$, and clearly the $\alpha$-values are increasing as the game evolves.

Similarly as in Section~\ref{sec:saturationlower}, we use the function $\phi(A,B,t)$ as defined in \eqref{eq:phidef} to measure the progress of the opponent, but this time the opponent is Maxi. The plan is to compare the $\alpha$-values with the $\phi$-values, for the purpose of finding a suitable edge for Mini's moves. Again, for all points in time $t \in \N_0$ and all $B,S \subseteq V$ we put
\[\mathcal{D}(B,S,t) := \big\{A \subseteq V \setminus B \mid \alpha(A,S,t) < \phi(A,B \setminus S,t)\big\}.\]
Notice that for all sets $B$ and $S$ we start with $\mathcal{D}(B,S,0)=\emptyset$. Finally, let
\[A_0(B,S,t) := \bigcap_{A \in \mathcal{D}(B,S(t),t)} A.\]

For all $t \in \N_0$ we have to define which sets $B=B(t)$ and $S=S(t)$ we want to use when comparing $\alpha$- and $\phi$-values. We choose $B(t)$ and $S(t)$ in such a way that $S(t) = B(t)$ as long as $|S(t)| < \ell$. As soon as $|S(t)|=\ell$, only $B(t)$ will further grow. We start with $S(0)=B(0)=\emptyset$. Whenever it is Maxi's turn at some point in time $t$, we put $S(t)=S(t-1)$ and $B(t)=B(t-1)$.  Whenever it is Mini's turn at round $t$, we first define the sets $S(t)$ and $B(t)$ and afterwards specify which edge she should play. It turns out that finding suitable sets $S(t)$ and $B(t)$ is quite tricky. Suppose at some point in time $t$ it is Mini's turn. We then advise Mini to run Algorithm~\ref{alg:findSandB} for appropriately defining $S(t)$ and $B(t)$.

\begin{algorithm}[H]
\caption{Algorithm for defining the sets $S(t)$ and $B(t)$}\label{alg:findSandB}
\begin{algorithmic}[1]
 \State $S := S(t-1)$, $B := B(t-1)$
 \If{$\mathcal{D}(B,S,t-1)=\emptyset$} $Z := V \setminus B$
 \Else $Z := A_0(B,S,t-1)$
 \EndIf
 \While{$B \neq V$ \textbf{and} $\alpha(Z,S,t-1) = |Z|\cdot|S|$}
  \If{$Z=V \setminus B$ \textbf{and} $|S|=\ell$} $B := V$
  \ElsIf{$\mathcal{D}(B,S,t-1)=\emptyset$}
   \State let $v \in V \setminus B$
   \State $S:=S \cup \{v\}$, $B:=B \cup \{v\}$ \Comment{Rule (U1)}
  \Else
   \If{$|S|=\ell$}
    \State $B := B \cup A_0(B,S,t-1)$ \Comment{Rule (U2)}
   \Else
    \State let $v \in A_0(B,S,t-1)$
    \State $S := S \cup \{v\}$, $B:=B \cup \{v\}$ \Comment{Rule (U3)}
   \EndIf
  \EndIf
  \If{$\mathcal{D}(B,S,t-1)=\emptyset$} $Z := V \setminus B$
  \Else $Z := A_0(B,S,t-1)$
  \EndIf
 \EndWhile
 \State $S(t) := S$, $B(t) := B$
 \State \textbf{return} $S(t)$, $B(t)$
\end{algorithmic}
\end{algorithm}

The idea behind Algorithm~\ref{alg:findSandB} is the following. In her move, Mini wants to play an edge $\{u,v\}$ such that $u \in V \setminus B(t)$ and $v \in S(t)$. In case $\mathcal{D}(B(t),S(t),t-1)$ is non-empty, we require $u \in A_0=(B(t),S(t),t-1)$. Hence, the set $Z$ used in the algorithm indicates from which set Mini will choose $u$ from. We observe that if all possible edges between $Z$ and $S$ are already present in $G(t-1)$, there is no edge $\{u,v\}$ available for Mini. We overcome this problem by repeatedly updating the sets $S$ and $B$. When applying rule (U1) or rule (U3), a new vertex $v$ is added to $S$. Note that $v$ is already connected to the other vertices of $S$, therefore for all points in time $t$, the vertices of $S(t)$ induce a complete graph in $G(t-1)$. In case $S$ already contains $\ell$ vertices, instead of adding a new vertex to $S$ we apply rule (U2) and transfer some nodes to the set $B$. We see that whenever $S(t) \neq B(t)$, we have $|S(t)|=\ell$ and between $S(t)$ and $B(t) \setminus S(t)$, all edges are present in $G(t-1)$. Note that a priori it is not clear whether the considered set $A_0(B,S,t-1)$ is non-empty and whether the algorithm terminates or not. 

We now continue by describing Mini's strategy for playing her edge at time $t$ that follows immediately after executing Algorithm~\ref{alg:findSandB}.

\begin{enumerate}
\item[(S1)] If $B(t)=V$ but $G(t-1)$ is not yet saturated, insert an arbitrary edge such that $G(t)$ is $k$-colorable.
\item[(S2)] If $\mathcal{D}(B(t),S(t),t-1)$ is empty and $B(t)\neq V$, let $u \in V \setminus B(t)$ and $v \in S(t)$ be two vertices such that the edge $\{u,v\}$ is not yet contained in $G(t-1)$. Insert $\{u,v\}$.
\item[(S3)] If $\mathcal{D}(B(t),S(t),t-1)$ is non-empty, let $u \in A_0(B(t),S(t),t-1)$ and $v \in S(t)$ be two vertices such that the edge $\{u,v\}$ is not yet contained in $G(t-1)$. Insert $\{u,v\}$.
\end{enumerate}

This finishes the definition of Mini's strategy. We now start proving Lemma~\ref{lem:stars}. Thereby, we also verify that Algorithm~\ref{alg:findSandB} terminates and that the proposed strategy is well-defined and covers all possible cases.

\begin{proof}[Proof of Lemma~\ref{lem:stars}.]

Let $n,\ell \in \N$ and let $k \ge 3\ell+1$. Suppose Mini applies the proposed strategy. We prove by induction that as long as $G(t-1)$ is not saturated, for Mini's move at time $t$ the following invariants hold.

\begin{enumerate}
\item[(I1)] Algorithm~\ref{alg:findSandB} always terminates and afterwards, exactly one rule of (S1)-(S3) can be applied.
\item[(I2)] Let Mini insert the desired edge, regardless whether $G(t)$ is $k$-colorable or not. Then the set $\mathcal{D}(B(t),S(t),t)$ is empty.
\item[(I3)] Mini can play her edge without violating the colorability constraint.
\item[(I4)] If there is a vertex $u \in V \setminus B(t)$ with $\alpha(u,S(t),t) < \ell$, then $G(t)$ is not yet saturated.
\end{enumerate}

We start with $B(0)=\emptyset$, $S(0) = \emptyset$, and $\mathcal{D}(B(0),S(0),0)=\emptyset$. Let $t$ be a point in time such that $G(t-1)$ is not saturated and assume that either it is Mini's first move of the game (providing the base case) or that by induction, (I1)-(I4) were true for all previous moves of Mini. We start with (I1). Recall that the set $Z$ used in Algorithm~\ref{alg:findSandB} indicates from which set Mini wants to pick a vertex $u$ and connect with a vertex $v \in S(t)$. With the criterion $\alpha(Z,S,t-1)=|Z|\cdot|S|$ in line 4 we test whether all these edges are already present in the graph $G(t-1)$.

First we assume that the set $\mathcal{D}(B(t-1),S(t-1),t-1)$ is empty. If this is the case, we put $B=B(t-1)$, $S=S(t-1)$, and in line 2 we set $Z = V \setminus B$. In the special case $V = B$ Algorithm~\ref{alg:findSandB} immediately stops and Mini can apply rule (S1). If $B\neq V$, we have three subcases. First, if $\alpha(Z,S,t-1) < |Z| \cdot |S|$, Algorithm~\ref{alg:findSandB} terminates, and rule (S2) can be applied. Next, if $\alpha(Z,S,t-1)=|Z|\cdot\ell$ and we put $B=V$ in line 5 (during the second iteration of the while-loop). Again the algorithm terminates, and Mini uses (S1) until the end of the game. It remains the case $\alpha(Z,S,t-1)=|Z|\cdot |S|$ but $|S|<\ell$. In this situation, when running Algorithm~\ref{alg:findSandB} we apply rule (U1) and add one vertex $v$ to the sets $S$ and $B$. We claim that after applying (U1), the set $\mathcal{D}(B,S,t-1)$ is still empty. Let $A \subseteq V \setminus B$ be any non-empty set (where we take the freshly updated set $B$ including $v$). Since $\mathcal{D}(B(t-1),S(t-1),t-1)$ was empty by assumption, we have 
\[\alpha(A,S(t-1),t-1) \ge \phi(A,B(t-1) \setminus S(t-1),t-1)=|E_{t-1}[A]|.\]
However, when adding $v$ to $S$ and $B$, clearly the $\alpha$-value of $A$ is non-decreasing while the $\phi$-value remains the same, so $A \notin \mathcal{D}(B,S,t-1)$. Hence after using rule (U1), $\mathcal{D}(B,S,t-1)$ is empty as claimed, $Z$ is set to $V \setminus B$, and we can repeat the whole argument for the case, potentially update the sets $S$ and $B$ several times, until either $\alpha(Z,S,t-1) < |Z| \cdot |S|$ or $\alpha(Z,S,t-1)=|Z|\cdot\ell$. Then indeed the algorithm terminates, and as discussed above Mini can apply either (S1) or (S2).

Let us now assume that the set $\mathcal{D}(B(t-1),S(t-1),t-1)$ is non-empty. Then in line 3 we put $Z=A_0(B(t-1),S(t-1),t-1) = \cap_{\mathcal{A} \in \mathcal{D}(B(t-1),S(t-1),t-1)}A$. We have to verify that this set is non-empty. Note that this case can only happen when $t > 1$. By induction we have $\mathcal{D}(B(t-2),S(t-2),t-2)=\emptyset$, so we can apply Lemma~\ref{lem:dangeroussets} with $f(v)=\alpha(v,S(t-2),t-2)$, $B=B(t-2) \setminus S(t-2)$, and $t_0=t-2$. Then the set $\mathcal{C}$ given by \eqref{eq:Doflemma} contains all subsets $A \subseteq V \setminus B(t-2)$ where $\alpha(A,S(t-2),t-2) < \phi(A,B(t-2) \setminus S(t-2),t-1)$. Since $S(t-1)=S(t-2)$, $B(t-1)=B(t-2)$, and $\alpha(v,S(t-1),t-1) \ge \alpha(v,S(t-2),t-2)$ holds for every vertex $v \in V$, it follows
\[\emptyset \neq \mathcal{D}(B(t-1),S(t-1),t-1) \subseteq \mathcal{C}.\]
By Lemma~\ref{lem:dangeroussets}~(i), $A' = \cap_{A \in \mathcal{C}}$ is non-empty and itself a member of $\mathcal{C}$, i.e., $\alpha(A',S(t-2),t-2) < \phi(A',B(t-2) \setminus S(t-2),t-1)$. Together with the assumption $\mathcal{D}(B(t-1),S(t-1),t-1) \neq \emptyset$ and statement (ii) of Lemma~\ref{lem:dangeroussets}, we conclude that for all $v \in A'$ it holds
\[\alpha(v,S(t-1),t-1) = \alpha(v,S(t-1),t-2).\]
Hence $A' \in \mathcal{D}(B(t-1),S(t-1),t-1)$, $A'=A_0(B(t-1),S(t-1),t-1)$, and therefore $A_0(B(t-1),S(t-1),t-1)$ is non-empty.

If at least one edge between $A_0(B(t-1),S(t-1),t-1)$ and $S(t-1)$ is not present in $G(t-1)$, Algorithm~\ref{alg:findSandB} immediately terminates, returns $S(t)=S(t-1)$ resp.\ $B(t)=B(t-1)$, and Mini can play her edge according to rule (S3). If all edges between $A_0(B(t-1),S(t-1),t-1)$ and $S(t-1)$ are already contained in $E_{t-1}$, there are two subcases.

Let us first investigate the subcase $|S(t-1)|=\ell$ where we are going to apply rule (U2) in line 11 and thus put $B=B(t-1) \cup A_0(B(t-1),S(t-1),t-1)$. (Note that we are still in the first iteration of the while-loop, so the ``old'' $B$ is the same as $B(t-1)$. Also note that here, we have $S=S(t-1)$.) We now apply Lemma~\ref{lem:dangeroussets}~(iii) and using $S(t-1)=S(t-2)$ and $B(t-1)=B(t-2)$, we see that for all $A \subseteq V \setminus (B(t-1) \cup A')$ it holds
\begin{align*}
&\alpha(A,S(t-1),t-1) \ge \alpha(A,S(t-2),t-2)\\ &\ge \phi(A \cup A',B(t-1) \setminus S(t-1),t-1)-\phi(A',B(t-1) \setminus S(t-1),t-1).
\end{align*}
Hence, after applying rule (U2), for each such set $A$ we have
\[\alpha(A,S,t-1) \ge \alpha(A,S(t-1),t-1) \ge \phi(A,B \setminus S(t-1),t-1) = \phi(A,B \setminus S,t-1),\]
and $\mathcal{D}(B,S,t-1)$ becomes empty. Consequently we put $Z=V \setminus B$ in line 15. Now either $\alpha(Z,S,t-1) < |Z| \cdot \ell$, Algorithm~\ref{alg:findSandB} terminates, and rule (S2) can be applied, or $\alpha(Z,S,t-1)=|Z| \cdot \ell$, we put $B=V$ in line 5 (during the second iteration of the while-loop), the algorithm terminates, and Mini uses (S1) until the end of the game.

In the subcase $|S(t-1)| < \ell$, we apply rule (U3), pick one vertex $v$ of the set $A_0(B(t-1),S(t-1),t-1)$, designate it as future ``star-vertex'' and add it to the set $S$. We claim that after executing rule (U3) in line 14, the set $\mathcal{D}(B,S,t-1)$ is empty. Indeed, by construction all sets $A \in \mathcal{D}(B(t-1),S(t-1),t-1)$ contain $v$, thus for every set $A \subseteq V \setminus (B(t-1) \cup \{v\})$ we have
\[\alpha(A,S(t-1),t-1) \ge \phi(A,B(t-1) \setminus S(t-1),t-1).\]
However, here it holds $B(t-1)=S(t-1)$ and $B=S$, thus
\[\phi(A,B \setminus S,t-1) = \phi(A,B(t-1) \setminus S(t-1),t-1).\]
It follows that $A \notin \mathcal{D}(B,S,t-1)$ as the $\alpha$-value of $A$ is non-decreasing. Hence $\mathcal{D}(B,S,t-1)$ is empty as claimed, and we consequently set $Z=V \setminus B$ in line 15. Now we are in the same situation as discussed above: either $\alpha(Z,S,t-1) < |Z| \cdot |S|$, or $\alpha(Z,S,t-1)=|Z| \cdot \ell$, or $\alpha(Z,S,t-1)=|Z| \cdot |S|$ but $|S| < \ell$. We have already seen that in all three subcases, at some point the algorithm stops and afterwards, either (S1) or (S2) serves as a matching rule. This proves invariant (I1).

We continue with the second invariant (I2). Suppose Mini plays the edge $e=\{u,v\}$. If Mini applied rule (S1), then $B(t)=V$ and the invariant is trivial. If Mini uses (S2), then $\mathcal{D}(B(t),S(t),t-1)$ is empty. W.l.o.g.\ we assume $u \in V \setminus B(t)$ and $v \in S(t)$. We observe that for no set $A \subseteq V \setminus B(t)$ its $\phi$-value can increase due to Mini's edge, so $\mathcal{D}(B(t),S(t),t)$ is empty as well.

It remains to check rule (S3). As we have seen before when analyzing Algorithm~\ref{alg:findSandB}, whenever Mini applies rule (S3) it holds $S(t)=S(t-1)=S(t-2)$ and $B(t)=B(t-1)=B(t-2)$, because all update rules (U1)-(U3) are set up such that the set $\mathcal{D}(B(t),S(t),t-1)$ becomes empty. Similarly as for rule (S2) we see that each set $A \subseteq V \setminus B(t)$ that is not contained in $\mathcal{D}(B(t-1),S(t-1),t-1)$ is also not contained in $\mathcal{D}(B(t),S(t),t)$. Notice that whenever (S3) is used we have $t >1$ and thus by induction $\mathcal{D}(B(t-2),S(t-2),t-2)=\emptyset$. We apply Lemma~\ref{lem:dangeroussets} with the same parameters as before and argue that for every set $A \subseteq V \setminus B(t-1)$ we have
\begin{align*}
\alpha(A,S(t-1),t-1) &\ge \alpha(A,S(t-2),t-2) \ge \phi(A,B(t-2) \setminus S(t-2),t-1)-1 \\ & = \phi(A,B(t-1) \setminus S(t-1),t-1)-1.
\end{align*}
Recall that the set $A'$ given by Lemma~\ref{lem:dangeroussets} is the same set as $A_0(B(t),S(t),t-1)$ and $\mathcal{D}(B(t-1),S(t-1),t-1) \subseteq \mathcal{C}$. Assume w.l.o.g.\ that Mini plays edge $e=\{u,v\}$ where $u \in A_0(B(t),S(t),t-1)$ and $v \in S(t)$. Then for all $A \in \mathcal{D}(B(t-1),S(t-1),t-1)$ we have $u \in A$. We see that with the new edge $e$, the $\phi$-value of $A$ does not increase whereas its $\alpha$-value increases by one, implying that $A \notin \mathcal{D}(B(t),S(t),t)$ as required. We deduce that $\mathcal{D}(B(t),S(t),t)$ must be empty, which veryfies (I2).

Regarding (I3), we first detect that for rule (S1) it is obvious that $G(t)$ is $k$-colorable. For rules (S2) and (S3), we argue that since $|S(t)|\le \ell$ by construction, it is sufficient to prove that the subgraph of $G(t)$ induced by $V \setminus S(t)$ is $(k-\ell)$-colorable. Recall that whenever $S(t) \neq B(t)$, it holds $|S(t)|=\ell$ in $G(t-1)$, every vertex of $B(t) \setminus S(t)$ is connected with each star vertex of $S(t)$, and the vertices of $S(t)$ induce a complete graph in $G(t-1)$. From these facts we infer that since the subgraph of $G(t-1)$ induced by $B(t)$ is $k$-colorable, the subgraph of $G(t-1)$ induced by $B(t) \setminus S(t)$ must be $(k-\ell)$-colorable. In the other case $S(t)=B(t)$, this is trivial. Further, we know that when Mini plays the edge $e=\{u,v\}$, both nodes $u$ and $v$ are contained in the set $(V \setminus B(t)) \cup S(t)$. Hence, in $G(t)$ the subgraph induced by $B(t) \setminus S(t)$ is $k-l$-colorable too. we now want to extend this property from $B(t) \setminus S(t)$ to the set $V \setminus S(t)$. By (I2), the set $\mathcal{D}(B(t),S(t),t)$ is empty, so for every non-empty set $A \subseteq V \setminus B(t)$ it holds
\begin{equation}\label{eq:kminuslcolor}
2 |E_t[A]|+|E_t[A,B(t)\setminus S(t)]| \le 2\phi(A,B(t) \setminus S(t),t) \le 2\alpha(A,S(t),t) \le 2\ell|A| < (k-\ell)|A|.
\end{equation}
Using Lemma~\ref{lem:coloring} we conclude that in $G(t)$, the subgraph induced by $V \setminus S(t)$ is $(k-\ell)$-colorable and thus $G(t)$ itself must be $k$-colorable.

We turn to the last invariant (I4) and assume that there exists $u \in V \setminus B(t)$ with $\alpha(u,S(t),t) < \ell$. Above when verifying invariant (I3) we have realized that the subgraph of $G(t)$ induced by $V \setminus S(t)$ is $(k-\ell)$-colorable. Hence if $S(t)$ contains less than $\ell$ vertices, it is obvious that $G(t)$ can not be saturated. So suppose $|S(t)|=\ell$. Then there exists a vertex $v \in S(t)$ such that $\{u,v\} \notin E_t$. We fictitiously assume that Maxi inserts the edge $f=\{u,v\}$ in her turn at time $t+1$. By invariant (I2) it holds $\mathcal{D}(B(t),S(t),t) = \emptyset$. Therefore Maxi playing edge $f$ corresponds to applying rule (S2) herself. Using similar arguments as above when analyzing rule (S2), we see that for no set $A$ the $\phi$-value increases due to edge $f$. Then $\mathcal{D}(B(t+1),S(t+1),t+1)$ is empty, given that Mini plays $f$. Then from \eqref{eq:kminuslcolor} for $t+1$ instead of $t$ and Lemma~\ref{lem:coloring} we deduce that $G(t+1)$ would be $k$-colorable. Therefore, $G(t)$ cannot be saturated and (I4) is true as well.

After verifying all four invariants by induction, we can now finish the proof of the lemma. Suppose Mini follows our strategy during the whole game and consider her turn at some point in time $t$. Her goal is to ensure that $G_{end}$ contains $\ell$ vertices of degree $n-1$. Suppose that after running Algorithm~\ref{alg:findSandB}, there exists at least one vertex $u \in V \setminus B(t)$ with $\alpha(u,S(t),t-1) < \ell$. As the strategy is well-defined, Mini answers by applying rule (S2) or rule (S3), and by (I3) indeed she can do so without violating the colorability constraint. Therefore, $G(t-1)$ was not saturated under our assumptions. However, by (I1) Algorithm~\ref{alg:findSandB} always terminates and one rule always applies. If $G(t-1)$ is saturated, then this must be rule (S1), in turn implying that $V=B(t)$. But then, all edges between $S(t)$ and $V \setminus S(t)$ are present in $G(t-1)$ and $|S(t)|=\ell$. Because the vertices of $S(t)$ induce a complete graph, we see that the vertices of $S(t)$ all have degree $n-1$ in $G(t-1)$.

On the other hand, if Mini makes $G(t)$ saturated with her edge played at time $t$, by (I4) we have $\alpha(V \setminus B(t),S(t),t)=|V \setminus B(t)| \cdot \ell$, i.e., $S(t)$ contains $\ell$ nodes and for all $u \in V \setminus B(t)$ it holds $\alpha(u,S(t),t)=\ell$. Then again the desired structure is present in $G(t)$. We see that no matter which player terminates the game, there are at least $\ell$ vertices of degree $n-1$ in $G_{end}$.
\end{proof}

\section{The Four Color Game}\label{sec:saturation4color}

In this section we study one specific saturation game by considering the special case $k=4$. By Theorem~\ref{thm:saturationlower} and Theorem~\ref{thm:saturationupper} we already have the bounds
\begin{equation}\label{eq:4colorsold}
\frac{n^2}{4} \le s(n,\chi_{>4}) \le \frac{n^2}{3}(1+o(1)).
\end{equation}
In particular, the upper bound follows because Mini has a strategy such that $G_{end}$ contains a vertex $s$ of degree $n-1$ (see Lemma~\ref{lem:stars}). The goal of this section is to prove Theorem~\ref{thm:4color} and thus close the gap in \eqref{eq:4colorsold}. We improve the lower bound by providing an alternative strategy for the specific game with $k=4$ which is more effective than the general strategy proposed in Section~\ref{sec:saturationlower}. In principle, the idea is the same as before: building up a collection of vertex-disjoint cliques ensures that no color class becomes too large. But in contrast to our previous strategy, we advise Maxi to proceed ``greedily'' and play the cliques within successive moves such that they cover the vertices that have been used most recently by Mini. Depending on Mini's strategy, Maxi answers by drawing cliques of size 4, 3, or 2 (that is, single edges).

After this informal description, let us proceed by introducing notations. On the one hand, our strategy for Maxi defines which edges she should play in her turns. On the other hand, it also describes how the collection of vertex-disjoint cliques grows during the process. At several points in time $t_i$, after the edge of round $t_i$ has been played Maxi defines a vertex set $V_i \subseteq V$ which then yields the $i$-th clique of the desired collection. More precisely, we require that the vertices of $V_i$ induce a complete graph in $G(t_i)$ and that all sets $V_i$ are vertex-disjoint. We then denote $t_i$ as the \emph{birth time} of the $i$-th clique. The sequence $t_i$ is non-decreasing, but subsequent cliques are allowed to have the same birth time. It turns out that it is convenient to also require that all birth times $t_i$ are chosen such that \emph{Mini} plays at time $t_i$. During the game process, Maxi only marks a finite number of cliques, and this number depends on Mini's strategy. Consequently, $V_i$ and $t_i$ are only defined for indices $i$ where the $i$-th clique actually exists. Finally, for each clique $V_i$ we put
\[W(i) := V \setminus \big(\bigcap_{j \le i} V_i\big).\]

Similarly as in Section~\ref{sec:saturationlower} and Section~\ref{sec:saturationupper}, we use the function $\phi$ as defined in \eqref{eq:phidef} to measure the progress of the opponent. For all $i \in \N$ where the $i$-th clique exists, we put
\[\phi(i) := \phi(W(i),V \setminus W(i),t_i) = |E_{t_i}[W(i)]|+|E_{t_i}[W(i),V \setminus W(i)]|.\]
One of Maxi's goals is to play such that the values $\phi(i)$ are globally bounded by a small constant, which means that we are always in the situation that her collection of vertex-disjoint cliques covers almost all non-isolated vertices of the current graph. 

We start by describing under which conditions it is reasonable for Maxi to create a triangle as next clique.

\begin{lemma}\label{lem:K3}
Let $i \in \N$ such that $\phi(i) \le 3$, $|W(i)| \ge 4$, and at most one edge of $E_{t_i}[W(i)]$ is isolated in $G(t_i)$. Then Maxi has a strategy for playing her edges and defining $V_{i+1}$ and $t_{i+1}$ such that
\begin{itemize}
\item[(i)] $t_{i+1} \le t_i+6$,
\item[(ii)] $|V_{i+1}|=3$, and
\item[(iii)] $\phi(i+1) \le 3$.
\end{itemize}
\end{lemma}

\begin{proof}
We prove the lemma by a case distinction. First suppose that in $G(t_i)$, the vertex set $W(i)$ already induces a triangle with vertices $x,y,z$. Then obviously we can put $V_{i+1} := \{x,y,z\}$ and $t_{i+1}:=t_i$. In particular, it then holds $\phi(i+1)=0$. 

Next assume that in $G(t-1)$ there exists a path $\{x,y,z\}$ where all three vertices are contained in $W(i)$. Then Maxi can play the edge $\{x,z\}$ at time $t_i+1$. After Mini played at time $t_i+2$, we put $V_{i+1}:=\{x,y,z\}$ and $t_{i+1}:=t_i+2$. We also have $\phi(i+1) \le 2$ since the two edges $\{x,y\}$ and $\{y,z\}$ that were contributing to $\phi(i)$ are not counted for $\phi(i+1)$.

Now we investigate the case where in $G(t_i)$, the graph induced by $W(i)$ is non-empty but neither contains a triangle nor a path of 3 vertices. Let $e=\{x,y\}$ be an arbitrary edge of $E_{t_i}[W(i)]$, and let $z \in W(i) \setminus \{x,y\}$ with maximal degree in $G(t_i)$. We then advise Maxi to complete the triangle $\{x,y,z\}$ within her next two moves at times $t_i+1$ and $t_i+3$. Note that in $G(t+3)$, there will be at most three edges between $V_{i+1} := \{x,y,z\}$ and $V \setminus V_{i+1}$: by testing all possible cases, we see that at most two edges of this type were already present in $G(t_i)$, and we have to add the edge that Mini eventually played at time $t_i+2$. Then it is easy to see that $G(t+3)$ is indeed $4$-colorable. Furthermore, $G(t+3)$ is not saturated due to the assumption $|W(i)| \ge 4$. So we put $t_{i+1}:=t_i+4$. Regarding property (iii), we observe that due to our choice of $z$, the assumption on $\phi(i)$ and the assumption that $E_{t_i}[W(i)]$ contains at most one isolated edge in $G(t_i)$, at most one edge that was counted in $\phi(i)$ can now contribute to $\phi(i+1)$. Adding Mini's edges at times $t_i+2$ and $t_i+4$ we arrive at $\phi(i+1) \le 3$. Note that here, we assumed that Mini does not help us in creating the triangle on the vertices $\{x,y,z\}$. In case Mini would play one of these edges on her own, it is easy to see that we obtain the same invariants already by picking $t_{i+1}:= t_i+2$.

Finally, we consider the remaining case where $E_{t_i}[W(i)]$ is empty. We pick $x,y \in W(i)$ such that $\deg(x)+\deg(y)$ is maximal in the graph $G(t_i)$. We first advise Maxi to play the edge $\{x,y\}$ at time $t_i+1$. After Mini's subsequent answer, at time $t_i+3$ Maxi picks $z \in W(i) \setminus \{x,y\}$ with maximal degree in $G(t_i+2)$. Then in her next two moves at times $t_i+3$ and $t_i+5$, Maxi completes the triangle $V_{i+1} := \{x,y,z\}$. Again, we assume for simplicity that Mini does not help Maxi in creating this triangle and plays different edges. We first have to verify that the proposed strategy is possible without violating the colorability constraint. Assume w.l.o.g.\ that in $G(t_i+5)$ we have $\deg(x) \ge \deg(y) \ge \deg(z)$. We observe that in this graph there can be at most five edges between $V_{i+1}$ and $V \setminus V_{i+1}$: at most three that were counted with $\phi(i)$, and at most two additional edges played by Mini at times $t_i+2$ and $t_i+4$. Therefore in $G(t_i+5)$ we have $\deg(y) \le 4$ and $\deg(z) \le 3$. The subgraph of $G(t_i+5)$ induced by $V \setminus \{y,z\}$ must be $4$-colorable because Maxi is not playing in this subgraph and Mini is also forced to maintain the colorability property. Since $|\Gamma_{t_i+5}(y)| \le 4$, $|\Gamma_{t_i+5}(z)| \le 3$, and $\{y,z\} \in E_{t_i+5}$, we can take an arbitrary coloring of $V \setminus \{y,z\}$ and extend it to the entire vertex set. Due to the assumption $|W(i)| \ge 4$, we also infer that $G(t_i+5)$ is not yet saturated. So we can take $t_{i+1} := t_i+6$. It remains to bound $\phi(i+1)$. Let 
\[E' := E_{t_i+2}[W(i)] \cup E_{t_i+2}[W(i),V \setminus W(i)].\]
The assumption $\phi(i) \le 3$ implies $|E'| \le 5$, but the edge $\{x,y\}$ is clearly not counted in $\phi(i+1)$. Because we are in the case $E_{t_i}[W(i)]=\emptyset$, the strategy now ensures that at most one edge of the set $E'$ can contribute to $\phi(i+1)$. At last, we take into account the two edges that Mini plays at times $t_i+4$ and $t_i+6$ and conclude that invariant (iii) is indeed satisfied.
\end{proof}

Next, we give a similar lemma for the special situation where Maxi is able to create a clique of size four during her next four moves.

\begin{lemma}\label{lem:K4}
Let $i \in \N$ such that $|W(i)| \ge 6$ and there exist at least two isolated edges $e,e' \in E_{t_i}[W(i)]$ that are isolated in $G(t_i)$. Then Maxi has a strategy for playing her edges and defining $V_{i+1}$ and $t_{i+1}$ such that
\begin{itemize}
\item[(i)] $t_{i+1} \le t_i+8$,
\item[(ii)] $|V_{i+1}|=4$, and
\item[(iii)] $\phi(i+1) \le \phi(i)+2$.
\end{itemize}
\end{lemma}
\begin{proof}
Suppose that the described situation occurs at time $t_i$. Then Maxi takes the two edges $e=\{x,y\}$ and $e'=\{x',y'\}$ and during her next four moves, she claims the four edges that remain such that the vertex set $V_{i+1} := \{x,x',y,y'\}$ induces a $K_4$. Since Mini can insert at most three additional edges meanwhile, it is not difficult to check that Mini has no possibility to forbid any of these four edges, so Maxi indeed succeeds in creating this $K_4$ until time $t_i+7$. Because we are assuming that $|W(i)|\ge 6$ and $e,e'$ have been isolated in $G(t_i)$, the graph $G(t_i+7)$ is not yet saturated. We then put $t_{i+1} := t_i+8$. For property (iii), we infer that $e$ and $e'$ contributed to $\phi(i)$ but don't contribute to $\phi(i+1)$. On the other hand, there are at most four edges that Mini played in meantime, therefore $\phi(i+1) - \phi(i) \le 2$. Noe that we assumed again that Mini does not help Maxi in creating the $K_4$. Otherwise, it is not difficult to see that (i)-(iii) can be already achieved with a smaller choice of $t_{i+1}$. 
\end{proof}

Finally, in some situations it is the best for Maxi to only play cliques of size $2$, i.e., simple edges. The purpose of such a strategy is to reduce the $\phi$-value and ensure that it is globally bounded during the game process.

\begin{lemma}\label{lem:K2}
Let $i \in \N$ such that $\phi(i) \le 5$ and $|W(i)| \ge 3$. Then Maxi has a strategy for playing her edges and defining $V_{i+1}$ and $t_{i+1}$ such that
\begin{itemize}
\item[(i)] $t_{i+1} \le t_i+2$,
\item[(ii)] $|V_{i+1}|=2$, and
\item[(iii)] $\phi(i+1) \le \max\{\phi(i)-1,1\}$.
\end{itemize}
\end{lemma}

\begin{proof}
First suppose that the edge set $E_{t_i}[W(i)]$ is non-empty. Then we can take an arbitrary edge $e=\{x,y\}$ of this edge set, put $V_{i+1} := \{x,y\}$ and $t_{i+1}=t_i$. Clearly, we then have $|E_{t_{i+1}}[W(i+1)]| < |E_{t_i}[W(i)]|$, which shows (iii) in this case.

If the set $E_{t_i}[W(i)]$ is empty, we take $x,y \in W(i)$ such that $\deg(x)+\deg(y)$ is maximal in $G(t_i)$, and tell Maxi to play the edge $e=\{x,y\}$ at time $t_i+1$. Assume w.l.o.g.\ that $\deg(x) \ge \deg(y)$. Then $\deg(y) \le 2$ by our assumption on $\phi(i)$. After inserting the edge $e$, clearly the subgraph of $G(t_i+1)$ induced by $V \setminus \{y\}$ is $4$-colorable, and since $\deg(y) \le 3$ in $G(t_i+1)$, every proper coloring of $V \setminus \{y\}$ can be easily extended to $y$. Since $|W(i)|\ge 3$ and  $E_{t_i}[W(i)]=\emptyset$, the strategy also ensures that $G(t_i+1)$ is not saturated. So we can put $V_{i+1} := \{x,y\}$ and $t_{i+1} := t_i+2$. It remains to prove (iii). We observe that if $\phi(i) \le 2$, the edge $e$ covers all edges that contributed to $\phi(i)$, hence the only edge that potentially counts for $\phi(i+1)$ is the edge that Mini plays at time $t_i+2$, thus $\phi(i+1) \le 1$. On the other hand, if $\phi(i) \ge 2$, the edge $e$ covers at least two edges that counted for $\phi(i)$, and we obtain $\phi(i+1) = \phi(i)-1$.
\end{proof}

After these preparations we can start proving Theorem~\ref{thm:4color}. Recall that we only have to provide a matching lower bound on $s(n,\chi_{>4})$.

\begin{proof}[Proof of Theorem~\ref{thm:4color}.]
Let $n \ge 5$. The proof strategy is to verify by induction that given the game process for the first $i$-cliques, we can apply one of the three auxiliary lemmas of this section to see that there exists a strategy for Maxi to create the next clique sufficiently fast. For the base case, we distinguish two cases regarding the identity of the starting player. If Maxi starts the game, she can draw an arbitrary triangle with vertex set $V_1$ during her first three turns, and we put $t_1 := 6$. No matter how Mini plays, we have $\phi(1) \le 3$. If Mini starts with the game, Maxi extend the first edge played by Mini to a triangle with vertex set $V_1$ during her first two moves. We then put $t_1 := 5$ and observe that $\phi(1) \le 2$. In addition, we observe here, the assumption $n \ge 5$ implies that $G(t_1)$ is not saturated as the smallest saturated graph requires at least seven edges.

We now claim that whenever Maxi has defined the $i$-th clique for some $i \in \N$ and $|W(i)| \ge 5$, then she also has a strategy to build the next clique such that one of the following four conditions is satisfied.
\begin{enumerate}
\item[(i)] $\phi(i+1) \le 3$ and $|V_{i+1}|=3$,
\item[(ii)] $\phi(i+1) \le 5$ and $|V_{i+1}|=4$,
\item[(iii)] $\phi(i+1) \le 4$ and $|V_{i}|=4$, or
\item[(iv)] $\phi(i+1) \le 3$ and $|V_{i-1}|=4$.
\end{enumerate}

We prove this claim by induction over $i$, so we assume by induction that $V_i$ and $\phi(i)$ satisfied one of the four conditions. Recall that $\phi(1)\le 3$ and $|V_1|=3$, so indeed $V_1$ and $\phi(1)$ serve as base case. The induction step now follows directly from Lemma~\ref{lem:K3}, Lemma~\ref{lem:K4}, and Lemma~\ref{lem:K2}. Let $i\ge 1$ and first suppose that $V_i$ and $t_i$ satisfied (i) or (iv). Then we distinguish two cases. If $E_{t_i}[W(i)]$ contains at least two edges which are isolated in $G(t_i)$, by Lemma~\ref{lem:K4} Maxi has a fast strategy to create the next clique such that $|V_{i+1}|=4$ and $\phi(i+1) \le 5$. Then $V_{i+1}$ and $\phi(i+1)$ fulfill (ii). In the other case where we don't have this pair of isolated edges, by Lemma~\ref{lem:K3} Maxi can play such that $|V_{i+1}|=3$, $\phi(i+1) \le 3$, and (i) is satisfied. Next suppose $V_{i}$ and $\phi(i)$ satisfy (ii). Here we let Maxi play such that the next clique is only a single edge (i.e.\ $|V_{i+1}|=2$). By Lemma~\ref{lem:K2}, Maxi is able to do so such that (iii) is true. Finally, in the case (iii), by assumption it holds $|V_{i-2}|=4$. Again, we require Maxi to play such that $|V_i|=2$, and then Lemma~\ref{lem:K2} establishes (iv).

We see that as long as $|W(i)| \ge 5$, the game does not stop and Maxi has a strategy to create at least one additional clique with vertex set $V_{i+1}$. Let $j$ be the unique index where $|W(j)| < 5$ and the procedure stops. At this point in time, we have $n' := \sum_{i=1}^{j} |V_i| > n-5$. For the remainder of the game, we let Maxi play arbitrarily until the graph is saturated and the game ends. We now prove that with the given strategy, Maxi ensures that $G_{end}$ contains at least $\frac{n^2}{3}+O(n)$ edges.

At time $t_j$, we have $V \setminus W(j) =\dot\cup_{i=1}^j V_i$. For $k \in \{2,3,4\}$, we denote by $a_k$ the number of sets $V_i$ in this collection with size $k$. By definition we have $a_2+a_3+a_4=j$. Moreover it holds
\begin{equation}\label{eq:coversum}
2a_2+3a_3+4a_4=n'>n-5.
\end{equation}
Next we observe that for every clique where $|V_i|=2$, $V_i$ and $\phi(i)$ either satisfy (iii) or (iv). Hence either $|V_{i-1}|=4$ or $|V_{i-2}|=4$, and we deduce
\[a_2 \le 2a_4.\]

No matter how Mini and Maxi play, $G_{end}$ is a complete $4$-partite graph and up to permutations, there exists exactly one proper $4$-coloring. Denote by $C_1, \ldots, C_4$ the four color classes. Clearly every $K_4$ of the collection contributes one vertex to each class. Next, every triangle of the collections spends one vertex to three different color classes. Let $m := |E_{t_{end}}|$ be the total number of edges in $G_{end}$. We observe that $m$ is minimal if the triangles always contribute to the same three color classes, say $C_1$, $C_2$, and $C_3$. Indeed, otherwise we could move one vertex from a smaller class to a higher class and forbid more edges. Finally, the same argument yields that from Maxi's perspective, in the worst case all $K_2$ of the collection account to the two heaviest color classes, say $C_1$ and $C_2$. Finally, we assume that the remaining $n-n'$ vertices that are not covered by the cliques all account for $C_1$. We summarize that $m$ becomes minimal if we have $|C_4|=a_4$, $|C_3|=a_3+a_4$, $|C_2|=a_2+a_3+a_4$, and $|C_1|=|C_2|+n-n'$. Using $n-n'=O(1)$, we obtain 
\begin{equation}\label{eq:edgesinGend}
m \ge |C_2|^2 + 2 \cdot |C_2| \cdot (|C_3|+|C_4|)+ |C_3| \cdot |C_4| + O(n).
\end{equation}

Minimizing this number subject to the boundary conditions on $a_2$, $a_3$, and $a_4$ is a standard optimization problem. For every fixed value $a_3$, the number of edges will be minimized when $a_2$ is maximal with respect to $a_4$ because this makes the color classes of $G_{end}$ as unbalanced as possible. Therefore, in the extremal case we have $a_2=2a_4$, which eliminates one variable. From \eqref{eq:edgesinGend} we then deduce 
\[m \ge 3a_3^2 + 17a_3a_4 + 22a_4^2 + O(n).\]
On the other hand, from \eqref{eq:coversum} it follows
\[a_2=\frac{n'-3a_3}{4}\quad\text{and}\quad a_4=\frac{n'-3a_3}{8},\]
and combining the last two (in-)equalities yields
\[m \ge \frac{1}{32}\left(11n'^2+2a_3n'-9a_3^2\right)+O(n).\]

The term $11n'^2+2a_3n'-9a_3^2$ is concave in $a_3$ and obtains its minimum at the boundary, i.e., when $a_3 \in \{0,n'/3\}$. Let us shortly compare the two cases. If $a_3=0$, we have $m \ge 11n^2/32+O(n)$, whereas $a_3=n'/3$ gives $m \ge n^2/3+O(n)$. We see that with the choice $a_3=n'/3$, the game stops earlier in the worst-case, and we argue that $a_3=n'/3$ minimizes the value of $m$, up to error term $O(n)$. We conclude that $m \ge n^2/3+O(n)$ and indeed, the score of this particular saturation game is $n^2/3+O(n)$.
\end{proof}

\begin{remark} \label{rem:4colors} Using the suggested strategies for Mini and Maxi we determined $s(n,\chi_{>4})=n^2/3+O(n)$. However, having the analysis on hand we can precisely describe the game process when both players follow an optimal strategy. During a first period of the game, by applying Lemma~\ref{lem:stars} Mini always uses the same vertex $s$ for her edges until $\deg(s)=n-1$. Hence $s$ becomes a star vertex. Meanwhile, Maxi uses the strategy provided with Lemma~\ref{lem:K3} and covers the leafs of this star greedily with a collection of triangles. The pace of both players is equal: both insert in total $n+O(1)$ edges for completing their tasks. This first phase ends at the moment where $\deg(s)=n-1$ and thus $s$ reserves one color class on its own. On the other hand, at the same time the vertex-disjoint triangles guarantee that the three other color classes are equally large. The score of the game is then determined, and the players spend the remaining time by filling the graph arbitrarily with edges until it becomes saturated and the game ends.
\end{remark}

\section{Concluding Remarks}\label{sec:saturationconcluding}

We described strategies for both Maxi and Mini that work for all parameters $k$ and turned out to be almost optimal and sufficiently strong for proving Theorem~\ref{thm:saturationlower} and Theorem~\ref{thm:saturationupper}. In Section~\ref{sec:saturation4color} we have seen that at least in the case $k=4$ it is possible to improve and refine Maxi's strategy such that the lower and upper bounds are matching. We think that Maxi's strategy can be further improved and that the bound given by Theorem~\ref{thm:saturationlower} is not optimal, but that it requires more advanced strategies to improve the lower bound. We believe that in general, it is challenging to determine the score of the colorability saturation game \emph{precisely}.

As discussed in the introduction, very little is known about the saturation game with respect to the property ``$G$ contains a copy of $H$'' where $H$ is a fixed subgraph, even for the choice $H=K_3$. One natural and very interesting example is the Hamiltonian saturation game where we pick $H=C_n$. It is conjectured that the score of the Hamiltonian game is $\Theta(n^2)$ \cite{hefetz2016saturation}. We hope that in near future, the understanding of this fascinating class of combinatorial games can be improved and some of the aforementioned specific games can be solved.

\bibliographystyle{plain}
\bibliography{refsnew}

\begin{thebibliography}{1}

\bibitem{biro2016upper}
Csaba Bir\'{o}, Paul Horn, and Jacob Wildstrom.
\newblock An upper bound on the extremal version of hajnal's triangle-free
  game.
\newblock {\em Discrete Applied Mathematics}, 198:20--28, 2016.

\bibitem{carraher2017game}
James~M. Carraher, William~B. Kinnersley, Benjamin Reiniger, and Douglas~B.
  West.
\newblock The game saturation number of a graph.
\newblock {\em Journal of Graph Theory}, 85(2):481--495, 2017.

\bibitem{erdos1964problem}
Paul Erd\H{o}s, Andr{\'a}s Hajnal, and John~W. Moon.
\newblock A problem in graph theory.
\newblock {\em American Mathematical Monthly}, 71(10):1107--1110, 1964.

\bibitem{faudree2011survey}
Jill~R. Faudree, Ralph~J. Faudree, and John~R. Schmitt.
\newblock A survey of minimum saturated graphs.
\newblock {\em The Electronic Journal of Combinatorics}, DS:no.\ 19, 2011.

\bibitem{fueredi1991hajnal}
Zolt\'an F\"{u}redi, Dave Reimer, and \'Akos Seress.
\newblock Hajnal's triangle-free game and extremal graph problems.
\newblock {\em Congressus Numerantium}, page no.\ 123, 1991.

\bibitem{hefetz2016saturation}
Dan Hefetz, Michael Krivelevich, Alon Naor, and Milo{\v{s}} Stojakovi{\'c}.
\newblock On saturation games.
\newblock {\em European Jorunal of Combinatorics}, 51(C):315--335, 2016.

\bibitem{lee2014new}
Jonathan~D. Lee and Ago~E. Riet.
\newblock New f-saturation games on directed graphs.
\newblock {\em Preprint available at arxiv:1409.0565}, 2014.

\bibitem{lee2015fsaturation}
Jonathan~D. Lee and Ago~E. Riet.
\newblock F-saturation games.
\newblock {\em Discrete Mathematics}, 338:2356--2362, 2015.

\bibitem{seress1992hajnal}
\'Akos Seress.
\newblock On hajnal's triangle-free game.
\newblock {\em Graphs and Combinatorics}, 8(1):75--79, 1992.

\end{thebibliography}
\end{document}